\documentclass[12pt,leqno]{article}
\usepackage{amsmath,amssymb,amsbsy,amsfonts,amsthm,latexsym,
         amsopn,amstext,amsxtra,euscript,amscd,amsthm,mathdots}

\allowdisplaybreaks

\numberwithin{equation}{section}

\newtheorem{lem}{Lemma}

\newtheorem{thm}{Theorem}

\begin{document}

\begin{large}
\centerline{\Large \bf Subsets of $\mathbb{F}^*_p$ with only small products or ratios}

\end{large}
\vskip 10pt
\begin{large}
\centerline{\sc  Patrick Letendre}
\end{large}
\vskip 10pt
\begin{abstract}
Let $p$ be a fixed prime. We estimate the number of elements of a set $A \subseteq \mathbb{F}^*_p$ for which
$$
s_1s_2 \equiv a \pmod{p} \quad \mbox{for some}\quad a \in [-X,X] \quad \mbox{for all}\quad s_1,s_2 \in A.
$$
We also consider variations and generalizations.
\end{abstract}


\vskip 20pt

\section{Introduction and notation}

Let $p$ be a fixed prime number. For any member $\alpha$ of an equivalence class of $\mathbb{Z}/p\mathbb{Z}$, we write
$$
|\alpha|_p:=\min_{k \in \mathbb{Z}}|\alpha+kp|
$$
and for any finite set $A$ we write $|A|:= \# A$ which should not be confused with the norm of a complex number. Inspired by the paper \cite{jc:mzg}, we are interested by the cardinality of a set $A \subseteq \mathbb{F}^*_p$ that satisfies some property. Precisely, for each $X \ge 1$ we let $\mathcal{S}(X)$ be the set of all subsets $A \subseteq \mathbb{F}^*_p$ that satisfy
\begin{equation}\label{property-1}
\Bigl|\frac{s_1}{s_2}\Bigr|_p \le X\ \mbox{and/or}\ \Bigl|\frac{s_2}{s_1}\Bigr|_p \le X\ \mbox{for each}\ (s_1,s_2) \in A^2.
\end{equation}
We thus define
$$
S(X):=\max_{\substack{A \in \mathbb{F}^*_p \\ A \in \mathcal{S}(X)}}|A|.
$$

Similarly, for each integer $n \ge 2$ and $X \ge 1$ we let $\mathcal{R}_n(X)$ be the set of all subsets $A \subseteq \mathbb{F}^*_p$ that satisfy
\begin{equation}\label{property-2}
|s_1 \cdots s_n|_p \le X\ \mbox{for all pairwise distinct}\ s_1,\dots,s_n \in A.
\end{equation}
Then, we consider the quantity
$$
R_n(X):=\max_{\substack{A \in \mathbb{F}^*_p \\ A \in \mathcal{R}_n(X)}}|A|.
$$

For any $m, n \in \mathbb{N}$, we write
$$
\tau_n(m):=|\{(d_1,\dots,d_n) \in \mathbb{N}^n:\ d_1\cdots d_n=m\}|.
$$
We will often use the well known fact that $\tau_n(m) \ll_{n,\epsilon} m^\epsilon$ for each $n \ge 2$ and $\epsilon > 0$. We also write $e_p(z):=\exp\bigl(\frac{2\pi i z}{p}\bigr)$ for any $z \in \mathbb{C}$.

\section{Statement of theorems}

\begin{thm}\label{thm:1}
For each $1 \le X \le \frac{p}{12}$, we have
$$
S(X) \ll_{\epsilon} \min \left(X^{\epsilon}+\frac{X^{2+\epsilon}}{p}, p^{1/2}\right)
$$
for each fixed $\epsilon > 0$.
\end{thm}

\begin{thm}\label{thm:2}
For each integer $n \ge 2$ and $1 \le X \le 0.24p$, we have
$$
R_n(X) \ll_{\epsilon,n} \min \left(X^{1/n+\epsilon}+\frac{X^{n/(n-1)+\epsilon}}{p^{1/(n-1)}},\min_{k \in \{2,\dots,n-2\}}X^{1/k+\epsilon}+\frac{X^{1+1/k+\epsilon}}{p^{1/k}}, p^{1/n+\epsilon}\right)
$$
for each fixed $\epsilon > 0$.
\end{thm}

\section{Preliminary lemmas}

There are a number of interesting results in the literature concerning multilinear exponential sums; see \cite{jb}, \cite{sm}, \cite{gp:iep} and \cite{ids} for example and other references. We will need the following two.

\begin{lem}\label{lem:1}
Let $A_1,\dots,A_n \subseteq \mathbb{F}^*_p$ $(n \ge 2)$ be subsets. Then
\begin{equation}\label{lem:1_eq}
\left|\sum_{a_1 \in A_1,\dots, a_n \in A_n} e_p(a_1\cdots a_n)\right| \le p^{1/2}(|A_1|\cdots |A_n|)^{\frac{n-1}{n}}.
\end{equation}
\end{lem}

\begin{proof}
We assume that $|A_1| \ge |A_2| \ge \cdots \ge |A_n|$. The inequality follows from the well known result
$$
\max_{m \in \mathbb{F}^*_p}\left|\sum_{a_1 \in A_1, a_2 \in A_2} e_p(m a_1 a_2)\right| \le (p|A_1||A_2|)^{1/2},
$$
see \cite[(2)]{gp:iep}.
\end{proof}

\begin{lem}\label{lem:2}
Let $0 < \delta < 1/4$ and $n \in \mathbb{Z}_+$. There is an effectively computable $\delta' = \delta'(\delta) > 0$ such that if $p$ is a sufficiently large prime and $A_1, \dots, A_n \subset \mathbb{F}_p$ satisfy\\
$
\begin{array}{ll}
(i) & |A_i| > p^\delta\ \mbox{for}\ 1 \le i \le n;\\
(ii) & \prod_{i=1}^n |A_i| > p^{1+\delta};
\end{array}
$\\
then there is the exponential sum bound
$$
\biggl|\sum_{a_1 \in A_1, \dots, a_n \in A_n} e_p(a_1 \cdots a_n)\biggr| < p^{-\delta'}|A_1| \cdots |A_n|.
$$
\end{lem}

\begin{proof}
It follows from Theorem A of the paper \cite{jb}.
\end{proof}

\section{Proof of Theorem \ref{thm:1}}

We assume throughout the proof that $A \in \mathcal{S}(X)$ and satisfies $S(X)=|A|$. We begin with the first inequality. We choose $s_1 \in A$ that realizes \eqref{property-1} with every element of $A$ by being at least $\frac{|A|}{2}$ times at the denominator. We denote by $A_1$ the set of values that are thereby at the numerator. Restricting our attention to $A_1$, we choose $s_2 \in A_1$ that realizes \eqref{property-1} with every element of $A_1$ by being at least $\frac{|A_1|}{2}$ times at the numerator and we denote by $A_2$ the set of values that are thereby at the denominator.

Now, for each value $s \in A_2$ we have two representations. Indeed,
$$
\frac{s}{s_1} \equiv a \pmod{p} \quad \mbox{and} \quad \frac{s_2}{s} \equiv b \pmod{p} \quad \mbox{with} \quad 0 < |a|,|b| \le X.
$$
We deduce that
$$
s_1 a \equiv \frac{s_2}{b} \pmod{p} \ \Rightarrow \ ab \equiv \frac{s_2}{s_1} \equiv
: \alpha \pmod{p} \quad \mbox{with} \quad 0 < |a|,|b|,|\alpha| \le X.
$$
We thus have $ab=\alpha+Kp$ with $0 \le |K| \le \bigl\lfloor\frac{2X^2}{p}\bigr\rfloor$. For each fixed value of $K$, the number of solutions $(a,b)$ is at most $2\tau_2(\alpha+Kp) \ll X^\epsilon$ and we deduce that
$$
|A| \le 4 |A_2| \ll X^\epsilon\left(1+\frac{X^2}{p}\right).
$$

We now turn to the second inequality. From Lemma \ref{lem:1} with $n = 2$, we know that
$$
|T_1|+|T_2|:=\left|\sum_{s_1,s_2 \in A} e_p\left(\frac{s_1}{s_2}\right)\right|+\left|\sum_{s_1,s_2 \in A} e_p\left(\frac{2s_1}{s_2}\right)\right| \le 2 p^{1/2}|A|.
$$
The result will follow if we can show that $|T_1|+|T_2| \ge \frac{|A|^2}{400}$. We will assume that $|T_1| \le \frac{|A|^2}{400}$. We denote by $W$ the set of pairs $(s_1,s_2) \in A^2$ for which we know from the hypothesis $A \in \mathcal{S}(X)$ that $0 < \bigl|\frac{s_1}{s_2}\bigr|_p \le X$. In particular, $|W|=\frac{|A|^2+|A|}{2}$. We divide $T_1$ into
$$
T_{1,1}+T_{1,2}+T_{1,3} := \sum_{\substack{s_1,s_2 \in A \\ (s_1,s_2) \in W}} e_p\left(\frac{s_1}{s_2}\right)+\sum_{\substack{s_1,s_2 \in A \\ |\frac{s_1}{s_2}|_p > \frac{3p}{8} \\ (s_1,s_2) \notin W}} e_p\left(\frac{s_1}{s_2}\right)+\sum_{\substack{s_1,s_2 \in A \\ |\frac{s_1}{s_2}|_p \le \frac{3p}{8} \\ (s_1,s_2) \notin W}} e_p\left(\frac{s_1}{s_2}\right).
$$
From now, we denote by $U_i$ the number of pairs $(s_1,s_2)$  in the sum $T_{1,i}$. We have $U_1=|W|=\frac{|A|^2+|A|}{2}$ and $U_2+U_3=\frac{|A|^2-|A|}{2}$ and we write $U_i:=u_i|A|^2$. From
$$
\Re (T_{1,1}+T_{1,2}+T_{1,3}) \le \frac{|A|^2}{400}
$$
and the hypothesis $1 \le X \le \frac{p}{12}$, we get to the inequality
\begin{eqnarray}\label{eq:1}
u_2 \ge u_1 \cos\left(\frac{\pi}{6}\right)-u_3 \cos\left(\frac{\pi}{4}\right)-\frac{1}{400}.
\end{eqnarray}
But we will have $|T_2| \ge \frac{|A|^2}{400}$ if
\begin{eqnarray}\label{eq:2}
u_1 \cos\left(\frac{\pi}{3}\right)+u_2 \cos\left(\frac{\pi}{2}\right) \ge u_3 + \frac{1}{400}.
\end{eqnarray}
We deduce from \eqref{eq:1}, $u_1=\frac{1}{2}+\frac{1}{2|A|}$ and $u_2+u_3=\frac{1}{2}-\frac{1}{2|A|}$ that 
\begin{equation}\label{eqthm1}
u_2 \ge \frac{\frac{1}{2}\left(\cos\left(\frac{\pi}{6}\right)-\cos\left(\frac{\pi}{4}\right)\right)+\frac{1}{2|A|}\left(\cos\left(\frac{\pi}{6}\right)+\cos\left(\frac{\pi}{4}\right)\right)-\frac{1}{400}}{1-\cos\left(\frac{\pi}{4}\right)}.
\end{equation}
Using inequality \eqref{eqthm1} and the fact that $\cos\left(\frac{\pi}{2}\right)=0$, we see that \eqref{eq:2} is satisfied if
$$
\left(1+\frac{1}{|A|}\right)\cos\left(\frac{\pi}{3}\right)+\frac{\cos\left(\frac{\pi}{6}\right)-\cos\left(\frac{\pi}{4}\right)+\frac{1}{|A|}\left(\cos\left(\frac{\pi}{6}\right)+\cos\left(\frac{\pi}{4}\right)\right)-\frac{1}{200}}{1-\cos\left(\frac{\pi}{4}\right)} \ge 1-\frac{1}{|A|}+\frac{1}{200}
$$
which holds since
$$
\cos\left(\frac{\pi}{3}\right)+\frac{\cos\left(\frac{\pi}{6}\right)-\cos\left(\frac{\pi}{4}\right)-\frac{1}{200}}{1-\cos\left(\frac{\pi}{4}\right)} \ge 1+\frac{1}{200}.
$$
The proof is complete.

\section{Proof of Theorem \ref{thm:2}}

We assume throughout the proof that $A \in \mathcal{R}_n(X)$ and satisfies $R_n(X)=|A|$. Also, for any $k \ge 1$, we say that $(s_1,\dots,s_k)$ is an {\it admissible} $k$-tuple if the $s_j$ are pairwise distinct $(j=1,\dots,k)$. There are exactly $|A|\cdot(|A|-1)\cdots(|A|-k+1)$ admissible $k$-tuples in $A^k$. We can assume that $|A|$ is large enough since otherwise there is nothing to prove.

We begin with the third inequality. Assuming that $X \le 0.24 p$ and that $|A| > p^{1/n+\delta}$ for some fixed $0 < \delta < 1/2n$, we get
{\small$$
|A|^n \ll \Biggl|\sum_{\substack{s_1,\dots,s_n \in A \\ (s_1,\dots,s_n)\ admissible}} e_p(s_1\cdots s_n)\Biggr|-O(|A|^{n-1}) \ll \Biggl|\sum_{s_1,\dots,s_n \in A} e_p(s_1\cdots s_n)\Biggr| \le p^{-\delta'}|A|^{n}
$$}\par
\noindent for some $\delta' > 0$, from Lemma \ref{lem:2}. This is a contradiction for $p$ large enough and we deduce that $|A| \ll p^{1/n+\epsilon}$ for each $\epsilon > 0$.

For the first inequality, we define $\alpha$ by
$$
\pm \alpha := \max_{\substack{r_1,\dots,r_n \in A \\ (r_1,\dots,r_n)\ admissible}} |r_1 \cdots r_n|_p,
$$
and we assume that $\alpha \equiv s_1 \cdots s_n \pmod{p}$ (with $(s_1,\dots,s_n)$ admissible). We now define a change of variable according to this choice. In the set $A':=A \setminus \{s_1,\dots,s_n\}$, we can write an element $r$ as $r \equiv a_j\frac{s_j}{\alpha} \pmod{p}$ for some $0 < |a_j| \le X$ $(j=1,\dots,n)$.

Any of the $|A'|\cdot(|A'|-1)\cdots(|A'|-n+1)$ admissible $n$-tuples $(r_1,\dots,r_n)$ gives rise to
\begin{eqnarray}\label{s1}
r_1\cdots r_n \equiv c \pmod{p} & \Rightarrow & a_1\frac{s_1}{\alpha}\cdots a_n\frac{s_n}{\alpha} \equiv c \pmod{p}\\ \nonumber
& \Rightarrow & a_1\cdots a_n \equiv c\alpha^{n-1} \pmod{p}\\ \label{s2}
& \Rightarrow & a_1\cdots a_n = c\alpha^{n-1} + Kp\\ \nonumber
\end{eqnarray}
where $0 < |c|,|a_1|,\dots,|a_n| \le X$ and $0 \le |K| \le \lfloor\frac{2X^n}{p}\rfloor$. From there, we distinguish two cases.

Case 1: $K=0$ for more than half of the admissible $n$-tuples. In this case, we have
\begin{eqnarray*}
|A'|^n & \ll & |\{(a_1,\dots,a_n,c) \in \mathbb{Z}^{n+1}:\ a_1\cdots a_n=c\alpha^{n-1},\ 0< |\alpha|,|c| \le X\}|\\
& = & 2^{n-1}\sum_{0 < |c| \le X}\tau_n(c\alpha^{n-1}) \ll X^{1+\epsilon}
\end{eqnarray*}
for each fixed $\epsilon \ge 0$.

Case 2: $K \neq 0$ for at least half of the admissible $n$-tuples. In this case, we fix a value of $r=r_1 \equiv a\frac{s_1}{\alpha}(\not\equiv 0) \pmod{p}$ that is in $\gg |A'|^{n-1}$ admissible $n$-tuples $(r,r_2,\dots,r_n)$ in \eqref{s1} that lead to \eqref{s2} with $K \neq 0$. Then, we consider the equation
\begin{eqnarray*}
rr_2\cdots r_n \equiv c \pmod{p} & \Rightarrow & aa_2\cdots a_n \equiv c\alpha^{n-1} \pmod{p}\\
& \Rightarrow & aa_2\cdots a_n = c\alpha^{n-1} + Kp\\
\end{eqnarray*}
with $0 < |c|,|a_2|,\dots,|a_n| \le X$ and $0 < |K| \le \lfloor\frac{2X^n}{p}\rfloor$. Now, we write $d:=gcd(a,\alpha^{n-1})$ and $a':=\frac{a}{d}$, $\beta:=\frac{\alpha^{n-1}}{d}$ and $K':=\frac{K}{d}$. We find that
$$
aa_2\cdots a_n = c\alpha^{n-1} + Kp\ \Rightarrow\ a'a_2\cdots a_n \equiv K'p \pmod{\beta}
$$
so that a fixed value of $K'$ gives at most $d$ values of $a_2\cdots a_n \pmod{\alpha^{n-1}}$. There are $\ll\frac{X^n}{dp}$ possible values for $K'$ and since $0 < |a_2\cdots a_n| \le \alpha^{n-1}$ we get that we have in fact at most $2d$ values of $a_2 \cdots a_n$ for each. That is, we have at most $\ll \frac{X^n}{p}$ possible values of $(c,K)$. We get
$$
|A'|^{n-1} \ll \sum_{\substack{(c,K)\\ c\alpha^{n-1}+Kp \neq 0}} \tau_{n-1}\left(\frac{c\alpha^{n-1}+Kp}{a}\right) \ll \frac{X^{n+\epsilon}}{p}
$$
for each fixed $\epsilon > 0$. For $n=2$ we have in fact $\epsilon=0$ in this last inequality. The result follows.

For the second inequality, let's write
{\small$$
r_k(a):=|\{(s_1,\dots,s_k) \in A'^k\ \mbox{admissible}:\ s_1 \cdots s_k \equiv a \pmod{p}\}|
$$}\par
\noindent for each $k=1,\dots,n-1$. For a fixed value of $k$ we can split each admissible $n$-tuple $(s_1,\dots,s_n) \in A'^n$ into $s_1 \cdots s_n \equiv bc \equiv a \pmod{p}$, i.e. respectively $s_1 \cdots s_{n-k} \equiv b \pmod{p}$ and $s_{n-k+1} \cdots s_n \equiv c \pmod{p}$. This leads to
\begin{eqnarray*}
|A'| \cdots (|A'|-n+1) & \le & \sum_{0 < |a| \le X}\sum_{b=1}^{p-1}r_{n-k}(b)r_{k}(|ab^{-1}|_p)\\
& \le & \max_{m \in \mathbb{F}^*_p} r_{k}(m)\sum_{0 < |a| \le X}\sum_{b=1}^{p-1}r_{n-k}(b)\\
& = & 2X|A'| \cdots (|A'|-n+k+1)\max_{m \in \mathbb{F}^*_p} r_{k}(m).
\end{eqnarray*}
Now, for any fixed $m \in \mathbb{F}^*_p$ we use the change of variable stated above to write
\begin{eqnarray*}
s_1 \cdots s_k \equiv m \pmod{p} & \Rightarrow & a_1 \cdots a_k \equiv \ell \pmod{p} \quad (\mbox{for some}\ \ell = \pm |\ell|_p )\\
& \Rightarrow & a_1 \cdots a_k = \ell + Kp \quad (\mbox{with}\ 0 \le |K| \le \biggl\lfloor\frac{2X^k}{p}\biggr\rfloor).
\end{eqnarray*}
As previously, we deduce that
$$
r_k(m) \le 2^{k-1}\sum_{K}\tau_k(\ell +  Kp) \ll X^{\epsilon}\biggl(1+\frac{X^k}{p}\biggr)
$$
Overall, we get to
$$
|A| \ll |A'| \ll \biggl(X^{1/k} + \frac{X^{1+1/k}}{p^{1/k}}\biggr)X^{\epsilon}
$$
for any $k=1,\dots,n-1$.

\section{Concluding remarks}

The set 
$$
A:=\{\pm 2^k: k = 0,\dots,\lfloor\log(X)/\log(2)\rfloor\}
$$
shows that $S(X) \gg \log(2X)$. Also, the set
$$
A:=\{\pm 1,\dots,\pm \lfloor X^{1/n} \rfloor\}
$$
shows that $R_n(X) \gg X^{1/n}$. We conjecture that both $S(X) \ll_\epsilon X^\epsilon$ and $R_n(X) \ll_{\epsilon,n} X^{1/n+\epsilon}$ hold for each $\epsilon > 0$ when $X \le \left(\frac{1}{2}-t\right)p$ for a fixed $t > 0$ as $p \rightarrow \infty$.

{\sc D\'epartement de math\'ematiques et de statistique, Universit\'e Laval, Pavillon Alexandre-Vachon, 1045 Avenue de la M\'edecine, Qu\'ebec, QC G1V 0A6} \\
{\it E-mail address:} {\tt Patrick.Letendre.1@ulaval.ca}

\end{document}